\documentclass[10pt]{article}
\usepackage{amsmath,amssymb,amsthm}
\usepackage[margin=1in]{geometry}
\usepackage{graphicx}
\usepackage{subcaption}
\usepackage[font=scriptsize]{caption}
\usepackage{color}
\usepackage{scrextend}
\usepackage{hyperref}

\newcommand{\RR}{\mathbb{R}}

\newcommand{\paren}[1]{\left( #1 \right)}
\newcommand{\brak}[1]{\left[ #1 \right]}
\newcommand{\pd}[2]{\frac{\partial#1}{\partial#2}}

\newcommand{\slp}{\mathcal{S}}
\newcommand{\dlp}{\mathcal{D}}
\newcommand{\bc}{\mathcal{B}}

\newcommand{\OO}{\mathcal{O}}
\newcommand{\Schrod}{Schr\"{o}dinger\ }
\newcommand{\bA}{{\bf A}}

\newtheorem{theorem}{Theorem}

\newtheorem{remark}{Remark}

\title{Transparent Boundary Conditions for the Time-Dependent \Schrod
Equation with a Vector Potential}

\author{Jason Kaye 
  \thanks{Courant Institute of Mathematical Sciences,
    New York University, New York, NY 10012.
    (Email: {jkaye@cims.nyu.edu}).}
  \and
  Leslie Greengard
  \thanks{Courant Institute of Mathematical Sciences,
    New York University, New York, New York 10012 and
    Flatiron Institute, Simons Foundation, New York, New York 10010.
    (Email: {greengard@cims.nyu.edu}).}
}

\date{}

\begin{document}

\maketitle

\begin{abstract}
We consider the problem of constructing transparent boundary conditions
for the time-dependent Schr\"odinger equation with a compactly
supported binding potential and, if desired, a spatially uniform,
time-dependent electromagnetic vector potential. Such conditions prevent
nonphysical boundary effects from corrupting a numerical solution in a
bounded computational domain. We use ideas from potential theory to
build exact nonlocal conditions for arbitrary piecewise-smooth
domains. These generalize the standard Dirichlet-to-Neumann
and Neumann-to-Dirichlet maps known for the equation in one dimension without a vector
potential. When the vector potential is included,
the condition becomes non-convolutional in time. For the one-dimensional
problem, we propose a simple discretization scheme and a fast algorithm to
accelerate the computation.

\vspace{.3in}

\hrule

\vspace{.1in}

This is the second version of the manuscript, correcting an important
omission: namely the paper \cite{ermolaev99}, which derived exact
transparent boundary conditions from Green's representation theorem, as
we do in Section \ref{pottheory} below.  We have made minor changes
to the presentation, emphasizing that the novelty of the present work is
in the derivation of a fast algorithm which makes the Green's
function-based approach practical.

\vspace{.1in}

\hrule
\end{abstract}

\section{Introduction}

We begin with the time-dependent \Schrod equation in free space, given by
\begin{equation}\label{eq:schrodfree}
  \begin{aligned}
    i u_t(x,t) &= - \Delta u(x,t) + i \bA(t) \cdot \nabla u(x,t) +
    V(x,t) \, u(x,t), 
    \quad & x \in \RR^d, \enspace t \in [0,T], \\
    u(x,0) &= u_0(x),  & \\
    |u(x,t)| &\to 0,   &|x| \to \infty.
  \end{aligned}
\end{equation}
Here, $u(x,t)$ represents a complex-valued wave function, 
$u_0(x)$ is the given initial data, and 
$V(x,t)$ is a binding potential, and
$\bA(t)$ is an applied electromagnetic vector potential.
We assume that $u_0(x)$ and $V(x,t)$ are supported
in a bounded domain $\Omega$ in $\RR^d$.
In order to solve 
\eqref{eq:schrodfree} on $\Omega$ alone, one must impose 
conditions on the boundary $\partial \Omega$ so that the solution of the
corresponding equation matches the solution to the free space
problem with sufficient precision. Such conditions are referred to by several 
names, including {\em absorbing boundary conditions}, {\em artificial boundary
conditions}, and {\em non-reflecting boundary conditions}.
Their purpose is typically to prevent spurious reflections of outgoing
waves from $\partial \Omega$ (but see Remark \ref{rem:theycomeback}
below). We refer to a boundary condition
\[\bc[u](x,t) = 0, \quad  x \in \partial \Omega\]
as an exact {\em transparent boundary condition} (TBC) when the solution of 
\begin{equation}\label{eq:schrodbvp}
  \begin{aligned}
    i u_t(x,t) &= - \Delta u(x,t) + i \bA(t) \cdot \nabla u(x,t) +
    V(x,t) \, u(x,t),
    \quad & x \in \Omega, \enspace t \in [0,T], \\
    \bc[u](x,t) &= 0,  & x \in \partial \Omega, \\
    u(x,0) &= u_0(x)  &
  \end{aligned}
\end{equation}
is equal to the solution of \eqref{eq:schrodfree}, restricted to
$\Omega$.

Exact TBCs for a
variety of linear and nonlinear wave propagation problems have been
studied extensively. They are
typically nonlocal in space and time, so that fast and memory-efficient algorithms are
required for their use in large problems. For the free space
\Schrod equation with $\bA(t) = 0$, TBCs are known in certain domains, and
several fast algorithms are available. For $d = 1$, Baskakov and Popov
\cite{baskakov91} derived an exact condition which takes the form of a
history-dependent Dirichlet-to-Neumann map and may be discretized as a
Robin condition. Fast algorithms to
implement this condition were introduced in
\cite{jiang04,lubich02,schadle06}. In higher dimensions, TBCs have been
derived for the circle, sphere, and half-space using a spectral
decomposition of the solution in space, and either a Laplace
transform in time or appropriate special functions 
\cite{hagstrom99,jiang08,han04,han07,antoine08}. Fast algorithms applicable to
these conditions are related to those developed for the one-dimensional
case \cite{lubich02,schadle06,jiang08}. A similar approach for general
domains was
introduced in \cite{schadle02}, but the conditions are given in the
Laplace transform domain, and efficiently recovering time domain
conditions is not straightforward. Half-space conditions were derived by a different
method in \cite{feshchenko11,feshchenko13,vaibhav15b}, and used to
assemble conditions for a rectangle in $d = 2$ and a box in $d = 3$.
These are written in terms of auxiliary unknowns, obtained by solving
lower-dimensional \Schrod equations on boundary faces and edges. Fast algorithms
for this approach were not considered. In another class
of methods, exact TBCs have been developed directly for particular 
discretizations of the \Schrod equation. Examples include
\cite{aes2003} for a waveguide geometry, which is closely related to the
one-dimensional case, \cite{arnold12} for the circle, and
\cite{ji18} for the rectangle. 

For the case $\bA(t) \neq 0$, there are fewer results.
In \cite{vaibhav15}, TBCs were presented for $d = 1$ with some
restrictions on the vector potential, and the fast
algorithm of \cite{lubich02} was adapted to the conditions. More
general TBCs for $d = 1$ were presented in
\cite{feshchenko17} without a fast algorithm.
In \cite{ermolaev99},
the full Green's function was derived for any dimension,
and it was observed that exact TBCs follow from Green's representation theorem,
although no fast algorithms were described.
We use this approach below
and show how exact
TBCs for \eqref{eq:schrodfree}, with or without a vector potential, are easily
obtained for
arbitrary piecewise-smooth domains in any dimension. 
A byproduct of the potential theory formulation
is that we obtain a representation of the solution which may be
evaluated anywhere in the exterior domain $\RR^d \backslash \Omega$, if
desired.  While the integral operators involved in the conditions are
convolutional in time when $\bA(t) = 0$, this is not true otherwise. 
For the one-dimensional case, we introduce
a new class of fast algorithms to apply these nonlocal,
non-convolutional operators.

Finally, we note that there are many methods used to prevent boundary reflections 
without implementing exact TBCs. Their purpose is to avoid the complexity
associated with exact nonlocal conditions, or to mimic an exact
condition when none is known. For the \Schrod equation, they include the
method of mask functions, perfectly matched layers, exterior complex
scaling, approximate absorbing boundary conditions based on
pseudodifferential calculus, and splitting methods.  We do not discuss
these approaches here, but refer the reader to
\cite{antoine08,antoine17,degiovannini15,weinmuller17,lorin09} and the
references therein.

\begin{remark}\label{rem:theycomeback}
  The vector potential term in \eqref{eq:schrodfree}
  represents an applied electromagnetic field, and causes advection of
  the solution. Therefore, depending on its form, this term may cause
  waves which have left the domain $\Omega$ to return later. This poses
  a challenge for the various approximate absorbing boundary conditions
  in common use, which, in essence, assume that waves crossing $\partial
  \Omega$ are always outgoing.  The TBCs we derive do not involve such
  an assumption, as will be demonstrated in our numerical results.
\end{remark}

The paper is organized as follows. 
In section \ref{pottheory}, we introduce the relevant aspects of potential
theory for the \Schrod equation 
and derive exact TBCs.
In section \ref{sec:implementation}, we describe a simple numerical
implementation of the exact condition in the one-dimensional case.
In section \ref{fastalg}, we discuss a fast algorithm for
evaluating the nonlocal, non-convolutional integral operators which appear in the 
TBC when the vector potential is present.
In section \ref{numres}, we illustrate the performance of our method
with several numerical examples, and in section \ref{conclusions}, we
mention a few directions for future work.

\section{Derivation of TBCs by potential theory} \label{pottheory}

In the absence of the binding potential and the vector potential, 
the wave function satisfies the free-particle \Schrod equation,
\begin{equation} \label{eq:freeparticle}
  i u_t = - \Delta u,
\end{equation}
whose Green's function \cite{economou} is given by 
\[K(x,t) = \frac{e^{i |x|^2 / 4t}}{({4 \pi i t})^{d/2}}.\]
Using this Green's function, we can define single and double layer
potentials with densities $\sigma$ and $\mu$, respectively, on a
spacetime boundary. In one dimension, these take the form
\begin{equation} \label{eq:slp1d}
  \slp[\sigma](x,t) := \int_0^t K(x-x_0,t-s) \, \sigma(s) \, ds = \frac{1}{\sqrt{4 \pi i}}
  \int_0^t \frac{e^{i (x-x_0)^2 / 4(t-s)}}{\sqrt{t-s}} \, \sigma(s) \, ds ,
\end{equation}
\begin{equation} \label{eq:dlp1d}
\dlp[\mu](x,t) := \int_0^t \frac{\partial{K}}{\partial x}(x-x_0,t-s) \, \mu(s) \, ds = 
\frac{\sqrt{i}}{4 \sqrt{\pi}}
  \int_0^t 
(x-x_0) \,
\frac{e^{i (x-x_0)^2 / 4(t-s)}}{(t-s)^{3/2}} \, \mu(s) \, ds ,
\end{equation}
for the boundary point $x = x_0$. The single layer potential is
continuous, and the double layer potential satisfies
the jump condition
\begin{equation} \label{kjump1d}
\lim_{x \to x_0^\pm} \dlp[\mu](x,t) = \pm \frac{i}{2} \mu(t).
\end{equation}
For $d>1$ and $\partial \Omega$ piecewise-smooth,
\begin{equation} \label{eq:slpnd}
    \slp[\sigma](x,t) := \int_0^t \int_{\partial \Omega} 
 K(x-y,t-s) \, \sigma(y,s) \, dS(y) \, ds,
\end{equation}
\begin{equation} \label{eq:dlpnd}
    \dlp[\mu](x,t) := \int_0^t \int_{\partial \Omega} 
\frac{\partial K}{\partial \nu_{y}}(x-y,t-s) \, \mu(y,s) \, dS(y) \, ds.
\end{equation}
Here and throughout, $\nu$ denotes a unit outward normal with respect to
the domain $\Omega$. The single layer potential is again continuous, and
the jump condition for the double layer potential is similar to that for the
parabolic case \cite{pogorzelski},
\begin{equation} \label{eq:kjump}
\lim_{\epsilon\rightarrow 0+}
\dlp[\mu](x_0 \pm \epsilon \nu_{x_0},t)
= \mp \frac{i}{2}\mu(x_0,t) + \dlp^\ast[\mu](x_0,t),
\end{equation} 
for $x_0 \in \partial \Omega$, excluding non-smooth boundary points, where
$\dlp^\ast[\mu]$ is the principal value of $\dlp[\mu]$.

When the vector potential is non-zero, the governing equation is
\begin{equation} \label{eq:schrodand}
  i u_t = - \Delta u + i \bA(t) \cdot \nabla u.
\end{equation}
In one dimension, $A(t)$ is a scalar and we have
\begin{equation} \label{eq:schroda}
  i u_t = - u_{xx} + i A(t) u_x.
\end{equation}
Let $\varphi(t)$ be the indefinite integral of $\bA(t)$ (or $A(t)$)
\[\varphi(t) = \int_0^t \bA(s) \, ds.\]
It is straightforward to verify that if $u(x,t)$ satisfies \eqref{eq:freeparticle}, 
then $u(x+\varphi(t),t)$
satisfies either \eqref{eq:schrodand} or \eqref{eq:schroda}, depending on the 
ambient dimension. 
The Green's function for \eqref{eq:schrodand},
\eqref{eq:schroda} is therefore given by
\begin{equation}\label{eq:greenfuna}
  K_A(x,t,s) = K(x+\varphi(t) - \varphi(s),t-s) = \frac{e^{i |x
  + \varphi(t) - \varphi(s)|^2 / 4(t-s)}}{(4 \pi i (t-s))^{d/2}},
\end{equation}
as shown in \cite{ermolaev99}.
We note that the advective nature of the vector potential term is
apparent from this viewpoint.

Single and double layer potentials $\slp_A[\sigma]$ and $\dlp_A[\mu]$
may be defined
as in \eqref{eq:slp1d}, \eqref{eq:dlp1d}, \eqref{eq:slpnd}, \eqref{eq:dlpnd},
with $K_A$ in place of $K$.
For $d > 1$,
the double layer potential satisfies the jump condition
\eqref{eq:kjump}, with $\dlp_A$ and $\dlp_A^\ast$ in place of
$\dlp$ and $\dlp^\ast$. $\dlp_A^\ast[\mu]$ again denotes 
the principal value of $\dlp_A[\mu]$.
For $d = 1$, the jump condition also includes the principle value term,
which was identically zero in \eqref{kjump1d}:
\begin{equation} \label{kajump1d} 
\lim_{x \to x_0^\pm} \dlp_A[\mu](x,t) = \pm \frac{i}{2}
\mu(t) + \dlp_A^\ast[\mu](t),
\end{equation}
with 
\begin{equation} \label{eq:dlp1dast}
  \dlp_A^\ast[\mu](t) =
    \frac{\sqrt{i}}{4 \sqrt{\pi}} \int_0^t \frac{\varphi(t)-\varphi(s)}{(t-s)^{3/2}}
  \exp\paren{i \, \frac{(\varphi(t)-\varphi(s))^2}{4(t-s)}} \, \mu(s) \, ds.
\end{equation}
The single layer potential $\slp_A[\sigma]$ is continuous in both cases.

\subsection{Green's identities and TBCs}

The following theorem gives TBCs for
\eqref{eq:schrodfree} on an arbitrary bounded domain $\Omega$ with a
piecewise-smooth boundary.
It can be found in slightly different form in 
\cite{ermolaev99}.

\begin{theorem} 
Let $x \in \RR^d \setminus \overline{\Omega}$ and suppose that $u$ satisfies the 
free particle \Schrod equation \eqref{eq:freeparticle} in this region, with zero initial
data. Then for $d > 1$, 
\begin{equation} \label{eq:greenform}
i u(x,t) = \slp_A \left[ \pd{u}{\nu} - i \nu \cdot \bA u \right](x,t)
- \dlp_A[u](x,t).
\end{equation}
If $x_0 \in \partial \Omega$, excluding non-smooth boundary points, we have
\begin{equation} \label{eq:green}
  \frac{i}{2} u(x_0,t) = \slp_A \left[\pd{u}{\nu} - i
  \nu \cdot \bA u \right](x_0,t) - \dlp^\ast_A[u](x_0,t).
\end{equation}
When $d = 1$, $\Omega = [-x_0,x_0]$, and $x > x_0$, we have 
\begin{equation} \label{eq:greenform1dr}
  i u(x,t) = \slp_A \left[ \pd{u}{x} - i A u
  \right](x,t) + \dlp_A[u](x,t) 
\end{equation}
and
\begin{equation} \label{eq:green1dr}
  \frac{i}{2} u(x_0,t) =  \slp_A \left[ \pd{u}{x} - i
  A u \right](x_0,t) + \dlp^\ast_A[u](t).
\end{equation}
For $x < -x_0$, we have
\begin{equation} \label{eq:greenform1dl}
  -i u(x,t) = \slp_A \left[ \pd{u}{x} - i A u
  \right](x,t) + \dlp_A[u](x,t)
\end{equation}
and
\begin{equation} \label{eq:green1dl}
  -\frac{i}{2} u(-x_0,t) = \slp_A \left[ \pd{u}{x} - i
  A u \right](-x_0,t) + \dlp^\ast_A[u](t),
\end{equation}
with $x_0$ replaced by $-x_0$ in the definitions of $\slp_A$ and
$\dlp_A$. Here $u$, $\pd{u}{\nu}$, and $\pd{u}{x}$ refer to their
appropriate spatial boundary traces when they are used as arguments to
$\slp_A$, $\dlp_A$, and $\dlp^\ast_A$, and $\nu$ is an outward normal on
$\partial \Omega$.
\end{theorem}

\begin{proof}

When $\bA \equiv 0$, or $A \equiv 0$, \eqref{eq:greenform},
\eqref{eq:greenform1dr}, and \eqref{eq:greenform1dl} are Green's
identities for the ordinary free-particle \Schrod equation. These are
well-known in the closely-related case of the heat equation
\cite{guentherlee1988}. Our proof follows the standard derivation of
this Green's identity, modified to include the vector potential term. We
only consider the case $d > 1$ here; the proof when $d = 1$ is similar. 

Let $s \in
[0,t]$ and $x,y \in \RR^d \setminus \overline{\Omega}$. Integrating \eqref{eq:schrodand} against
$K_A(x-y,t,s)$ on $[0,t] \times \RR^d \setminus \overline{\Omega}$ gives
\[0 = \int_0^t \int_{\RR^d \setminus \overline{\Omega}} K_A(x-y,t,s) \,
  \brak{i u_s(y,s) +
  \Delta u(y,s) - i \bA(s) \cdot
\nabla u(y,s)} \, dV(y) \, ds = T_1 + T_2 + T_3,\]
where we have split the expression into three terms. Integrating $T_1$ by parts in $s$ gives
\[T_1 = -\int_0^t \int_{\RR^d \setminus \overline{\Omega}} i \partial_s K_A(x-y,t,s) \, u(y,s) \,
  dV(y) \, ds + \int_{\RR^d \setminus \overline{\Omega}} i
\brak{K_A(x-y,t,s) \, u(y,s)}_{s=0}^t \, dV(y) .\]
In the first integral, we can compute $\partial_s K_A(x-y,t,s)$ using
\eqref{eq:greenfuna}. In the second integral, the boundary term at $s = 0$ vanishes since $u(x,0) = 0$ in $\RR^d
\setminus \overline{\Omega}$. For the boundary term at $s = t$, we use
the $\delta$-function property of the Green's function. We obtain
\[T_1 = \int_0^t \int_{\RR^d \setminus \overline{\Omega}} \brak{i \,
\pd{K}{t}(\star,t-s) + i \, \bA(s) \cdot \nabla
K(\star,t-s)} \, u(y,s) \, dV(y) \, ds + \, i \, u(x,t),\]
where we have used the symbol $\star := x - y + \varphi(t) - \varphi(s)$.
For $T_2$, we use Green's second identity and \eqref{eq:greenfuna} to
obtain
\[T_2 = \int_0^t \int_{\RR^d \setminus \overline{\Omega}} \Delta K(\star,t,s) \, u(y,s) \,
   dV(y) \, ds + \int_0^t \int_{\partial \Omega} \pd{K_A}{\nu_y}(x-y,t,s) \, u(y,s) -
   K_A(x-y,t,s) \, \pd{u}{\nu}(y,s) \, dS(y) \, ds.\]
For $T_3$, the divergence theorem gives
\[T_3 = -\int_0^t \int_{\RR^d \setminus \overline{\Omega}}  i \,
  \bA(s) \cdot \nabla K(\star,t,s) \,  u(y,s) \,dV(y) \, ds + \int_0^t \int_{\partial
  \Omega} i \, \nu_y \cdot \bA(s) \, K_A(x-y,t,s) \, u(y,s) \, dS(y) \,
ds.\]
We now add $T_1$, $T_2$, and $T_3$, and use the fact that $K$ satisfies the free-particle
\Schrod equation \eqref{eq:freeparticle}. After some algebra, we
find
\[0 = i \, u(x,t) + \int_0^t \int_{\partial
  \Omega} \pd{K_A}{\nu_y}(x-y,t,s) \, u(y,s) - K_A(x-y,t,s) \,
  \paren{\pd{u}{\nu}(y,s) - i \, \nu_y
\cdot \bA(s) \, u(y,s)}\, dS(y) \, ds\]
which is \eqref{eq:greenform}. \eqref{eq:green} follows from the jump
condition \eqref{eq:kjump} and the continuity of the single layer
potential by taking the limit as $x$ approaches a smooth boundary point $x_0$.
\end{proof}

Since we have assumed that $u_0$ and $V$ are supported inside $\Omega$,
the identities \eqref{eq:green}, \eqref{eq:green1dr}, and
\eqref{eq:green1dl} hold for \eqref{eq:schrodfree} and are exact 
TBCs. To borrow terminology
from \cite{baskakov91}, these may be viewed as generalized Robin
conditions, relating the Dirichlet and Neumann data at a time $t$ to
those throughout the time history $0 < s < t$. TBCs for the \Schrod
equation in arbitrary piecewise-smooth domains
with $\bA(t) = 0$ are obtained as a special case. Furthermore,
\eqref{eq:greenform}, \eqref{eq:greenform1dr}, and
\eqref{eq:greenform1dl} serve as representation formulas for $u$ outside of
$\Omega$, and may be evaluated there if needed.

\begin{remark}
When $d = 1$ and $A(t) = 0$, $\dlp_A^\ast \equiv \dlp^\ast$ vanishes and
\eqref{eq:green1dr} may be written in the form of the standard
Neumann-to-Dirichlet (NtD) map:
\begin{equation} \label{eq:n2dclassical}
u(x_0,t) =  \frac{e^{-3 \pi i/4}}{\sqrt{\pi}} \int_0^t
  \frac{u_x(x_0,s)}{\sqrt{t-s}} \, ds.
\end{equation}
The Dirichlet-to-Neumann (DtN) map may be recovered by viewing
\eqref{eq:n2dclassical} as an Abel integral equation \cite{pogorzelski}
and solving it explicitly. It is given by
\begin{equation} \label{eq:d2nclassical}
  \pd{u}{x}(x_0,t) = \frac{e^{3 \pi i/4}}{\sqrt{\pi}} \int_0^t
  \frac{u_t(x_0,s)}{\sqrt{t-s}} \, ds.
\end{equation}
Therefore the TBCs above generalize the known DtN and
NtD maps, which are derived, for example, in \cite{baskakov91,jiang04}
by other means. The conditions for the higher-dimensional cases and/or 
when $\bA(t) \neq 0$ may similarly
be written explicitly as DtN or NtD maps involving the inversion of 
Volterra integral operators. However, as we show in the next section,
one can simply discretize the conditions as written to obtain ordinary 
Robin conditions.

\end{remark}

\section{Implementation of the TBCs in one dimension}
\label{sec:implementation}

To see how the TBCs presented in the previous section may be
used in practice, we describe a simple second-order accurate
scheme for the case $d=1$. The method presented
here may be extended to a higher-order discretization in time, and is
independent of the choice of spatial discretization in the interior of
the domain.

For the right boundary $x = x_0$ of the computational domain, we must discretize
\eqref{eq:green1dr}. We subdivide $[0,T]$ into $N$ equispaced time
intervals, and define $t_n = n \Delta t$ with $\Delta t = T/N$. To achieve
second-order accuracy, we use a piecewise linear approximation of the solution.
In this approximation a function of interest $f(t)$ may be written as
\[f(t) = \sum_{n=1}^N f_n \eta_n(t)\]
with $f_n = f(t_n)$, where $\{ \eta_1,\ldots,\eta_N \}$ is the 
standard basis of ``hat" functions,
\[\eta_n(t) = \max(1-|t-t_n|/\Delta t,0).\]
Note that $\eta_n(t)$ is supported in $[t_{n-1},t_{n+1}]$.
Approximating $u$ and $\pd{u}{x}$ in this way, we obtain a
discretization of \eqref{eq:green1dr},
\begin{equation}
\label{green:discrete}
\paren{\frac{i}{2} I_N + i S_N A_N - D_N} {\bf u}_N - S_N {\bf v}_N = 0. 
\end{equation}
Here, $I_N$ is the $n \times n$ identity matrix, $A_N$ is a diagonal
matrix with $A_N(n,n) = A(t_n)$, 
\[ {\bf u}_N = [u_1,u_2,\dots,u_N] \approx
[u(t_1),u(t_2),\dots,u(t_N)], \]
and
\[ {\bf v}_N = [v_1,v_2,\dots,v_N] \approx
\left[ \pd{u}{x}(t_1),\pd{u}{x}(t_2),\dots,\pd{u}{x}(t_N) \right].  \]
$S_N$, $D_N$ are dense lower triangular matrices with 
\begin{align}
S_N(m,n) &:= 
\frac{1}{\sqrt{4 \pi i}} \int_0^{t_m} \frac{\exp\paren{i \,
(\varphi(t_m) - \varphi(s))^2 / 4 (t_m-s)}}{\sqrt{t_m-s}} \, \eta_n(s) \, ds \\
&=
\frac{1}{\sqrt{4 \pi i}} \int_{\max(0,t_{n-1})}^{\min(t_m,t_{n+1})} 
\frac{\exp\paren{i \,
(\varphi(t_m) - \varphi(s))^2 / 4 (t_m-s)}}{\sqrt{t_m-s}} \, \eta_n(s) \, ds, 
\end{align}
and
\begin{align}
  D_N(m,n) &:= \frac{\sqrt{i}}{4\sqrt{\pi}} \int_0^{t_m}
  \frac{\varphi(t_m)-\varphi(s)}{(t_m-s)^{3/2}}
  \exp\paren{i \, \frac{(\varphi(t_m)-\varphi(s))^2}{4(t_m-s)}} \, \eta_n(s) \, ds \\
 &= \frac{\sqrt{i}}{4\sqrt{\pi}} \int_{\max(0,t_{n-1})}^{\min(t_m,t_{n+1})}
  \frac{\varphi(t_m)-\varphi(s)}{(t_m-s)^{3/2}}
  \exp\paren{i \, \frac{(\varphi(t_m)-\varphi(s))^2}{4(t_m-s)}} \, \eta_n(s) \, ds. 
\end{align}
Note that the matrix entries are local in time and can be precomputed
using any suitable quadrature scheme. Because of the lower triangular structure
of these matrices (a consequence of the fact that the integral operators are of
Volterra type), we may write the discrete Green's identity
\eqref{green:discrete} as
\begin{align} \label{eq:sigmam}
  \paren{\frac{i}{2} + i \, S_N(m,m) \, A_m - D_N(m,m)} u_m &- S_N(m,m)
  \, v_m =  \nonumber \\
& \sum_{n=1}^{m-1} \brak{D_N(m,n) \, u_n + S_N(m,n) \, (v_n -
  i \, A_n \, u_n)},
\end{align}
for $m = 1,\ldots,N$. These are inhomogeneous Robin boundary conditions
for each time step $t_m$, with the right-hand side involving boundary data from previous time
steps.

\subsection{Butterfly scheme for the rapid evaluation of the TBCs} \label{fastalg}

The cost of computing the Robin coefficients in \eqref{eq:sigmam} naively is
of the order $\OO(m)$ at the $m$th time step, and $\OO(N^2)$ in total. Thus, 
without a fast algorithm to apply $S_N$ and $D_N$, computing the boundary
conditions will dominate the asymptotic cost of a simulation in the limit of many time steps.

To develop a fast algorithm, we
first partition the matrices $S_N$ and $D_N$ into blocks as in Figure
\ref{fig:triblk}, refining blocks towards the diagonal until the dimension of the
smallest blocks is a small constant. Rather than applying one row of the
matrix per time step, as suggested by the right hand side of
\eqref{eq:sigmam}, we can apply each block as soon as the corresponding
entries of ${\bf u}_N$ and ${\bf v}_N$ become available in the course of
time-stepping. The order in which the blocks may be applied is indicated by the
numbering in the figure. Each row of the triangular sections near the matrix
diagonal (marked in the figure by asterisks) contains at most a constant number of elements, and may be
built and applied directly. The results of these matrix-vector and
row-vector products may then be arranged and added together as needed to
compute the right hand side of \eqref{eq:sigmam}. All that is required to overcome
the $\OO(N^2)$ cost is a suitable method to apply each of the square blocks efficiently.

\begin{figure}[ht]
  \centering
  \includegraphics[width=0.4\textwidth]{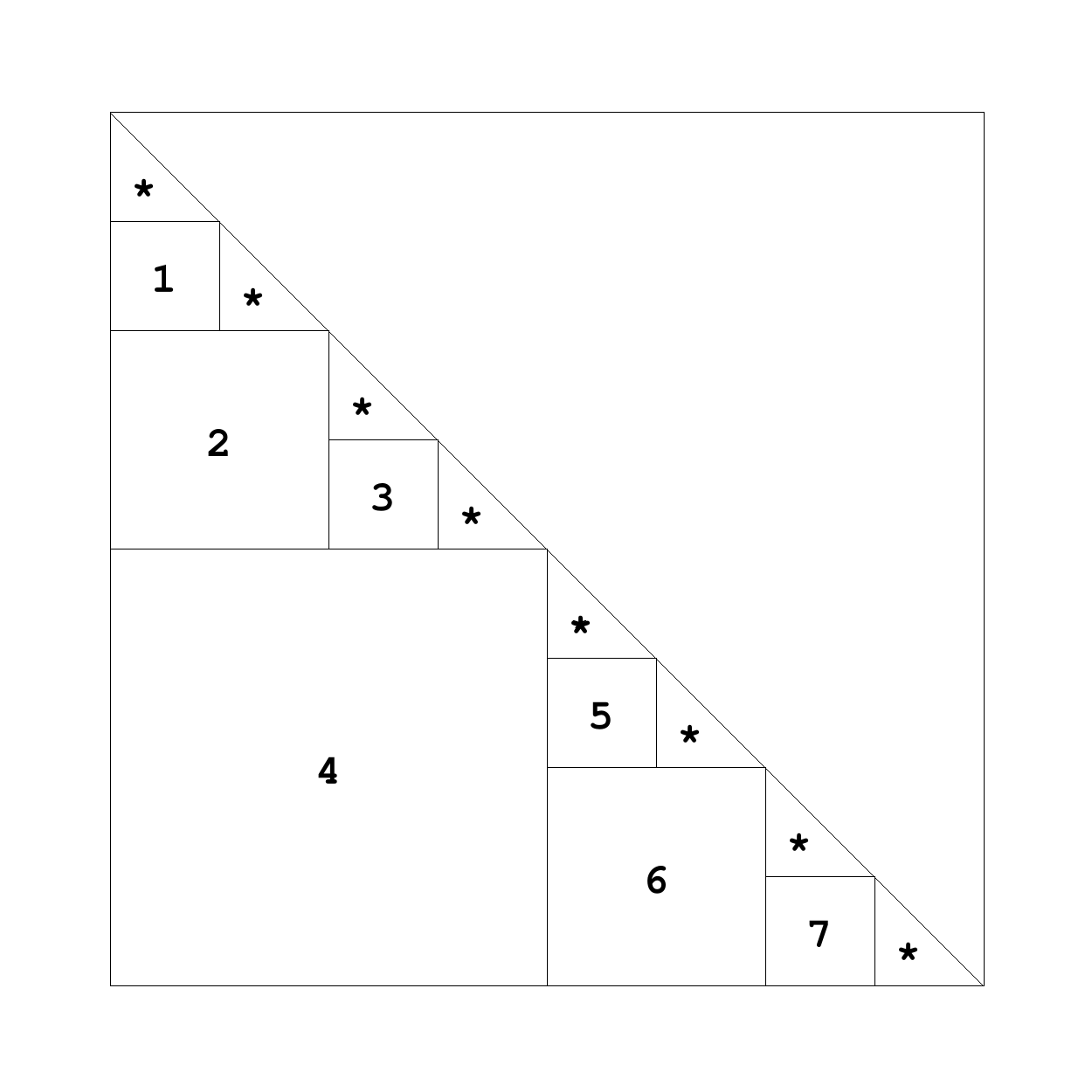}

  \caption{Block partitioning of the matrices $S_N$ and $D_N$.
  Numbered blocks are butterfly compressed in advance, and applied in the order indicated. Triangular sections
marked by $\star$ are applied directly.}
  \label{fig:triblk}
\end{figure}

If the integral operators were convolutional, the fast Fourier transform (FFT) could be used to 
apply each block in nearly optimal time, as in \cite{hairer85}.
In the present case, we replace the FFT with 
hierarchical {\em butterfly compression}
\cite{michielssen,oneil10}, which is related to the FFT and applies to
more general matrices.
The only property such matrices require is the 
so-called ``butterfly property'' - that the rank of a submatrix is
proportional to its area. Butterfly algorithms
have been successfully used to compress and
apply matrices arising from a variety of special function transforms and
oscillatory integral operators. 
We will not review the literature here, and instead refer the reader
to \cite{oneil10,candes09,hoying}. Since the matrices $S_N$ and $D_N$ are
discretizations of an oscillatory integral operator with a smooth phase
function, the butterfly algorithm should apply in its standard form.

For each $n \times n$ matrix, the algorithm of \cite{oneil10} begins with an
$\OO(n^2)$ precomputation step, in which the matrix is compressed as a sparse
factorization with only $\OO(n \log n)$ nonzero elements. Afterward, the matrix
may be applied at a cost of $\OO(n \log n)$. We use this algorithm for
each block. Summing the costs from large to small blocks, we find the
total precomputation cost is
\[\frac{N}{2} + 2 \paren{\frac{N}{4}}^2 + 4
\paren{\frac{N}{8}}^2 + \cdots \sim \OO(N^2).\]
The total matrix apply and memory storage costs are
\[\frac{N}{2} \log \frac{N}{2} + 2 \paren{\frac{N}{4} \log
  \frac{N}{4}} + 4 \paren{\frac{N}{8} \log \frac{N}{8}} + \cdots \sim \OO(N
\log^2 N),\]
since there are $\OO(\log N)$ terms in the sum. The
precomputation for $S_N$ and $D_N$ must be performed once for
each choice of $A(t)$, $\Delta t$, and $T$, but does not depend on $V$,
$u_0$, $x_0$, or any spatial discretization parameters. The scheme is
therefore  particularly efficient when one wishes, for example, to fix $A(t)$ and
solve the \Schrod equation for multiple initial conditions $u_0$ and
potentials $V$.

The butterfly ``compressibility" of various classes of
matrices is an area of ongoing research, and we will not
undertake a theoretical analysis of the rank properties of 
$S_N$ and $D_N$ here. The reader will find discussions of
this topic for several classes of matrices in the aforementioned
references. All of our numerical experiments indicate that $S_N$ and
$D_N$ are butterfly compressible even for large-amplitude vector potentials,
and that the algorithm performs efficiently and with the expected scaling, as
will be demonstrated in the following section.

\section{Numerical examples}\label{numres}

To demonstrate the effectiveness of our transparent boundary
condition in the one-dimensional case with $A(t) \neq 0$, we discretize
\eqref{eq:schrodbvp}, with $\Omega = [-x_0,x_0]$, by a
Crank-Nicolson scheme coupled to the Robin condition \eqref{eq:sigmam}.
The resulting method is second-order accurate in the time step $\Delta
t$ and the grid spacing $\Delta x$.

\subsection{Example 1: Gaussian wavepacket with \texorpdfstring{$V =
0$}{V = 0}}\label{sec:numex1}

For our first example, we set $x_0 = 1$, $V = 0$, and take $u_0$
to be a Gaussian wavepacket,
\begin{equation} \label{eq:gauswave}
  u_0(x) = \frac{1}{\sqrt{\alpha}} e^{ik(x-\mu)} e^{-(x-\mu)^2/4
  \alpha^2},
\end{equation}
with $\alpha = 0.08$, $k = -10$, and $\mu = 0$, so that the support of $u_0$ is
contained in $\Omega$ to at least fourteen digits of accuracy. 
The potential $A(t)$ is taken to be a pulse
\begin{equation} \label{eq:afield}
  A(t) = A_0 \sin^2(t \pi/T) \cos(\omega t)
\end{equation}
with $A_0 = 3000$, $\omega = 300$, and $T = 0.1$. The indefinite
integral $\varphi$ may be computed analytically. The wavepacket is
advected approximately $\max |\varphi| \approx 10$ domain
radii from the origin at its maximal
excursion, sweeping it back and forth across the domain in several cycles.
Figure \ref{fig:ex1images} shows the absolute value of a numerical solution obtained
using the TBCs, along with $A(t)$. A video of the solution is available
at the webpage:
\url{https://cims.nyu.edu/~kaye/kg_tbcse1_ex.html}.

\begin{figure}[ht]

  \hspace*{\fill}
  \includegraphics[height=10cm]{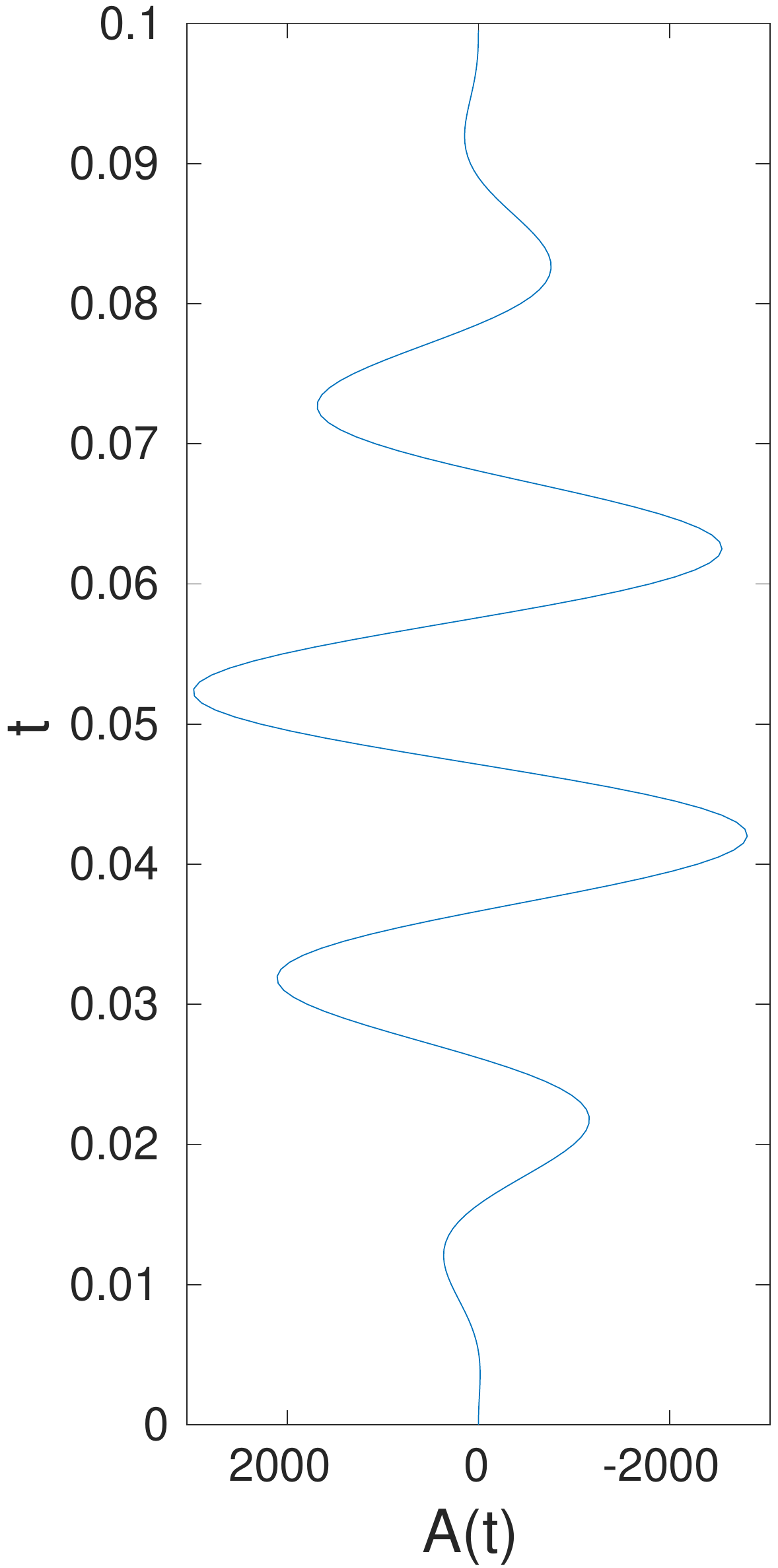}
  \includegraphics[height=10cm]{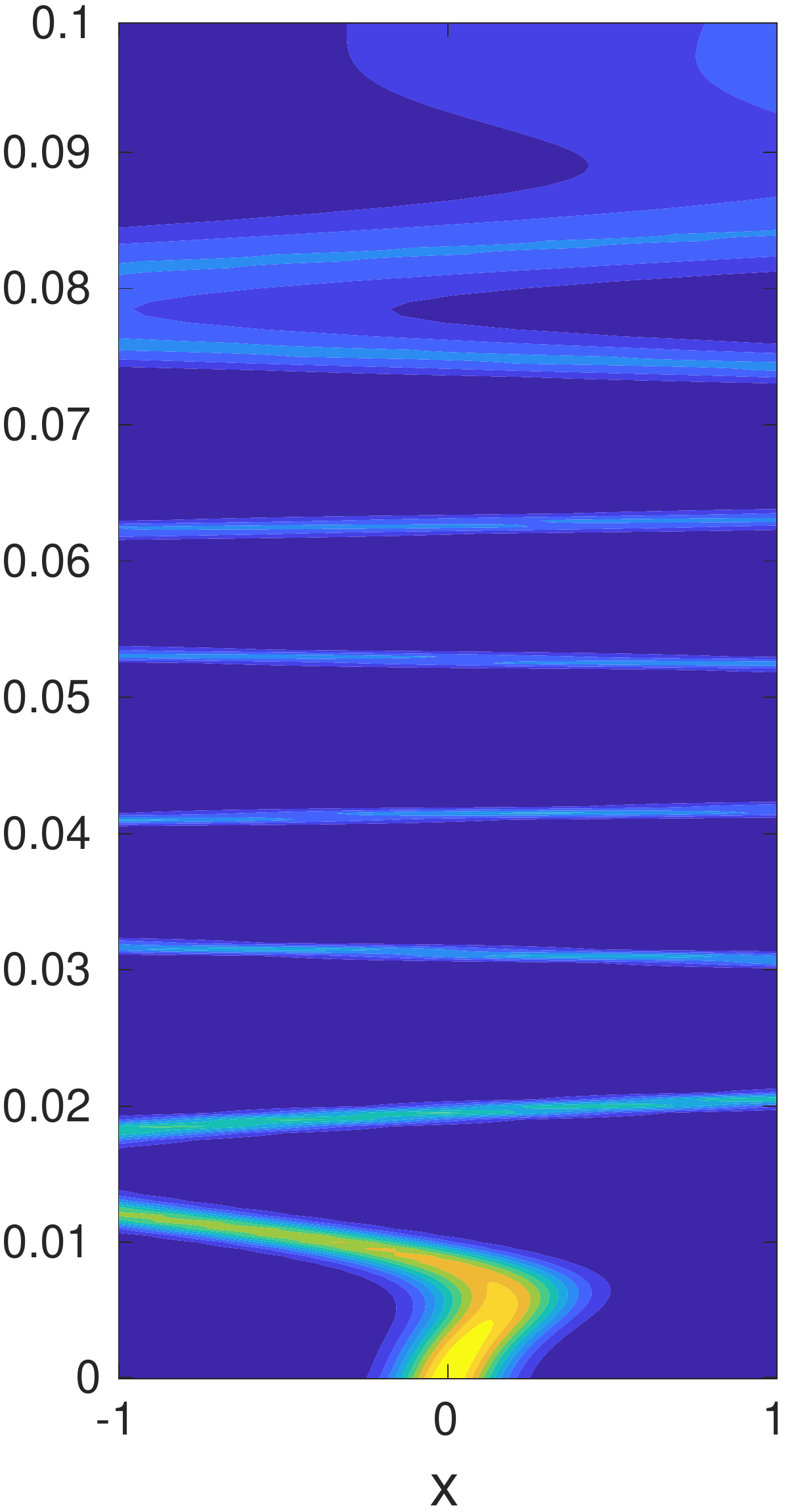}
  \hspace*{\fill}

  \caption{Plots of $A(t)$ and $|u(x,t)|$ for Example 1, with time on the vertical axis. The
  horizontal axis has been reversed in the first plot so that the sign
  of $A(t)$ lines up with the direction of advection. As the
wavepacket disperses, it is pushed rapidly back and forth
across the domain by the applied field, causing it to leave and return repeatedly.
At its maximal excursion, the center of the wavepacket is approximately
10 domain radii from the origin.}
  \label{fig:ex1images}
\end{figure}

Since the solution of the free-particle \Schrod equation with Gaussian
wavepacket initial data is known analytically, the solution for our
case may be obtained by shifting it by $\varphi$. We can therefore
measure the error directly. 
We first demonstrate that the entire discretization scheme is second-order in $\Delta
t$. The scheme is also second-order in $\Delta x$, but we focus on the
convergence in $\Delta t$ since this includes the discretization error of the boundary
condition. We fix $\Delta x$ sufficiently small so that the influence
of the spatial discretization is eliminated, and compute the numerical
solution for several choices of $\Delta t$. For each one, we compute the
maximum $L^2$ error on $[-1,1]$ over the duration of the simulation. The results,
given in Table \ref{tab:ex1errvsdt}, demonstrate the desired convergence
rate.

In order to isolate any discretization error caused by the boundary
condition, we solve the PDE using the Crank-Nicolson scheme on the
domain $[-20,20]$ with the Robin condition replaced by zero Dirichlet
boundary conditions. These are correct to machine
precision for this case. We use the same fixed $\Delta x$ as before, and again measure
the maximum $L^2$ error on $[-1,1]$ for the same choices of $\Delta t$.
The results, also given in Table \ref{tab:ex1errvsdt}, show that for the
same $\Delta t$, the TBC scheme is actually more accurate than the brute
force scheme.

\begin{table}[ht]
  \centering
  \begin{tabular}{|c|r|r|r|r|r|}
    \hline
    $\Delta t$ & $1 \times 10^{-5}$ & $5 \times 10^{-6}$ & $2.5 \times
    10^{-6}$ & $1.25 \times 10^{-6}$ & $6.25 \times 10^{-7}$ \\ \hline
    TBC & $2.62 \times 10^{-1}$ & $6.61 \times 10^{-2}$ &
    $1.67 \times 10^{-2}$ & $4.21 \times 10^{-3}$ & $1.09 \times
    10^{-3}$ \\ \hline
    DBC & $7.29 \times 10^{-1}$ & $2.12 \times 10^{-1}$ &
    $5.43 \times 10^{-2}$ & $1.37 \times 10^{-2}$ & $3.49 \times 10^{-3}$ \\
    \hline
  \end{tabular}
  \caption{Maximum $L^2$ error of the Gaussian wavepacket solution on
  $[-1,1]$ for recursively halved values of $\Delta t$ and fixed
  $\Delta x = 2 \times 10^{-4}$, using the transparent boundary condition scheme
(TBC) as well as homogeneous Dirichlet boundary conditions on a much larger domain
(DBC).}
  \label{tab:ex1errvsdt}
\end{table}

We next measure the total time required to compute the TBCs, first by
applying rows of $S_N$ and $D_N$ directly, and then by using the
butterfly scheme described in Section \ref{fastalg}. The matrices $S_N$
and $D_N$ are precomputed and butterfly compressed beforehand.  Table
\ref{tab:times} shows the time spent on finite difference marching and
the time spent computing boundary conditions.  The cost of obtaining
TBCs using the butterfly scheme appears to scale sublinearly with $N$
for the values tested, and represents a negligible part of the total
cost of the simulation. With the direct scheme, this cost scales like
$O(N^2)$, and eventually overtakes the cost of marching. 

\begin{table}[ht]
  \centering
  \begin{tabular}{|l|r|r|r|r|r|r|}
    \hline
    \multicolumn{1}{|c|}{$N = T/\Delta t$} & 10,000 & 20,000 & 40,000 & 80,000 & 160,000 
    \\ \hline
    Time for finite difference marching & 4.9 & 9.7 & 19.4 & 39.0 &
    78.2 \\ \hline
    Time to obtain TBCs (direct) & 0.42 & 1.68 & 6.81 & 28.02 &
    120.90 \\ \hline
    Time to obtain TBCs (butterfly) & 0.17 & 0.29 & 0.42 & 0.66 &
    1.25 \\ \hline
  \end{tabular}
  \caption{Wall clock timings, in seconds, for Crank-Nicolson 
marching, and for
    obtaining TBCs with and without the butterfly scheme.
Recursively doubled values of $N$ correspond to the values of $\Delta
t$ shown in Table \ref{tab:ex1errvsdt}.}
  \label{tab:times}
\end{table}

\begin{remark}
  As mentioned above, the timings in Table \ref{tab:times} do not include the
  construction and/or compression of $S_N$ and $D_N$. In the direct application
  scheme, the cost of building these matrices is $\OO(N^2)$. The memory
  required to store them is also $\OO(N^2)$, so for sufficiently large
  $N$ one must build them on the fly. The cost of this step cannot be
  shared over many simulations, and would be a significant addition to
  the costs reflected in the table.  In the butterfly scheme, the cost of building and compressing the matrices is
  $\OO(N^2)$, but the memory required to
  store them is only $\OO(N \log^2 N)$. Therefore the matrices may be
  built, compressed, and stored, once for each choice of $T$, $A(t)$,
  and $\Delta t$, thereby eliminating any online cost associated with matrix
  construction.
\end{remark}

\subsection{Examples 2-4: Gaussian wavepacket interacting with a repulsive potential}

Our second set of examples demonstrates the case $V \neq 0$, and shows that 
the same precomputed and compressed integral operators may be used 
for a series of simulations with several
different choices of $V$. We again take $u_0$ to be a Gaussian wavepacket
\eqref{eq:gauswave}, now with $\alpha = 0.08$, $k = 0$, and $\mu = -2$. We define $A(t)$ as in
\eqref{eq:afield}, with $A_0 = -220$, $\omega = 0$, and $T = 0.1$. We then let $V$ be a Gaussian centered at the origin:
\[V(x) = V_{\max} e^{-x^2/2\beta^2},\]
with $\beta = 0.1$. $V_{\max}$ is taken to be $4000$, $6000$, or 
$8000$. We choose $\Omega = [-3,3]$, so that $V$ and $u_0$ are zero to
machine precision outside of $\Omega$.

In each example, the Gaussian wavepacket is advected
to the right and interacts with the potential barrier
$V$. Some of the incident wavepacket is
reflected, and some transmitted, depending on $V_{\max}$. 
The reflected and transmitted waves then exit the domain
through the transparent boundaries. The absolute value of each solution
and the vector potential $A(t)$ are shown in Figure \ref{fig:ex2images}.
Videos of each solution are available at the URL mentioned in Section
\ref{sec:numex1} above.

\begin{figure}[ht]
  \hspace*{\fill}
  \includegraphics[height=7.5cm]{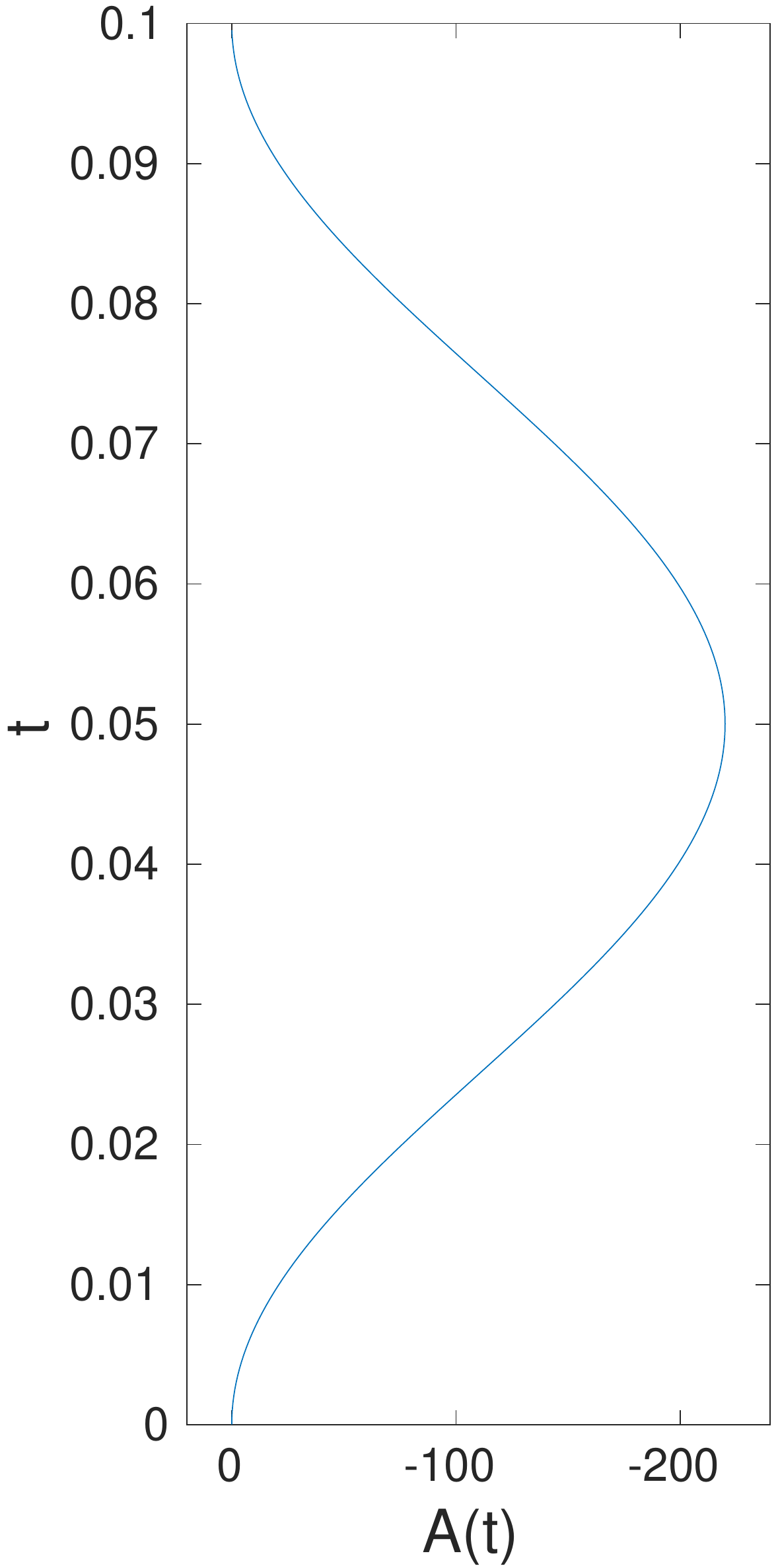}
  \includegraphics[height=7.5cm]{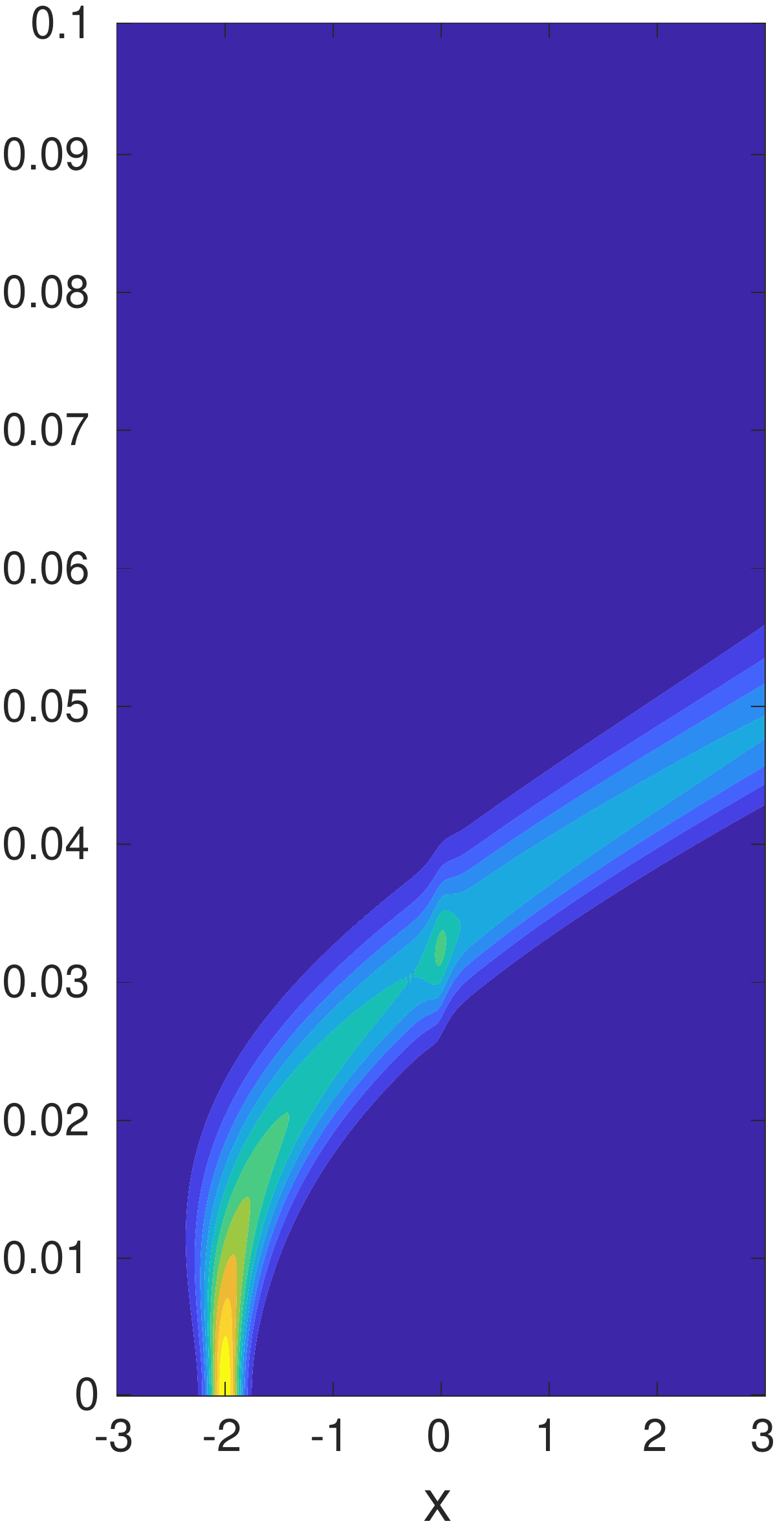}
  \includegraphics[height=7.5cm]{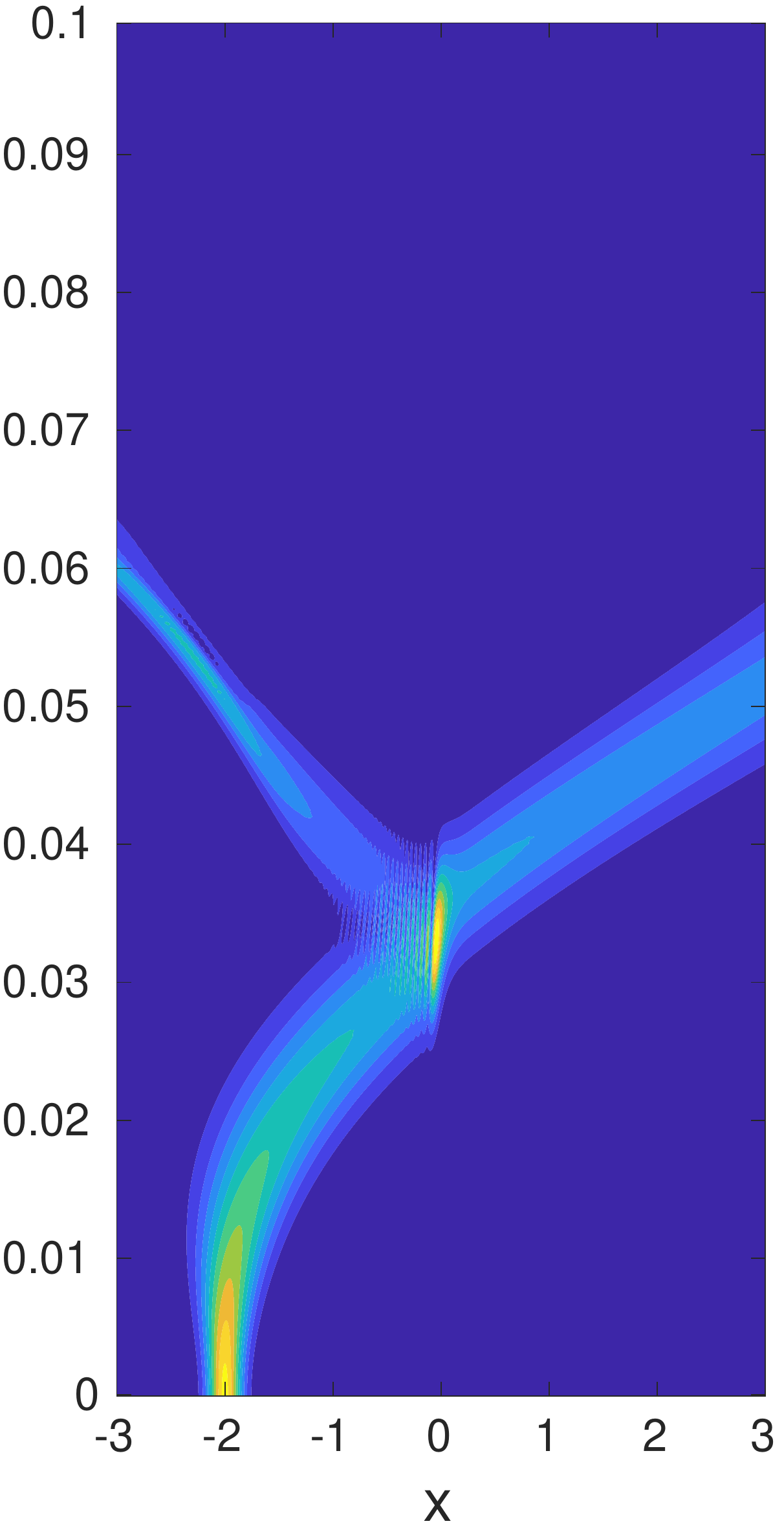}
  \includegraphics[height=7.5cm]{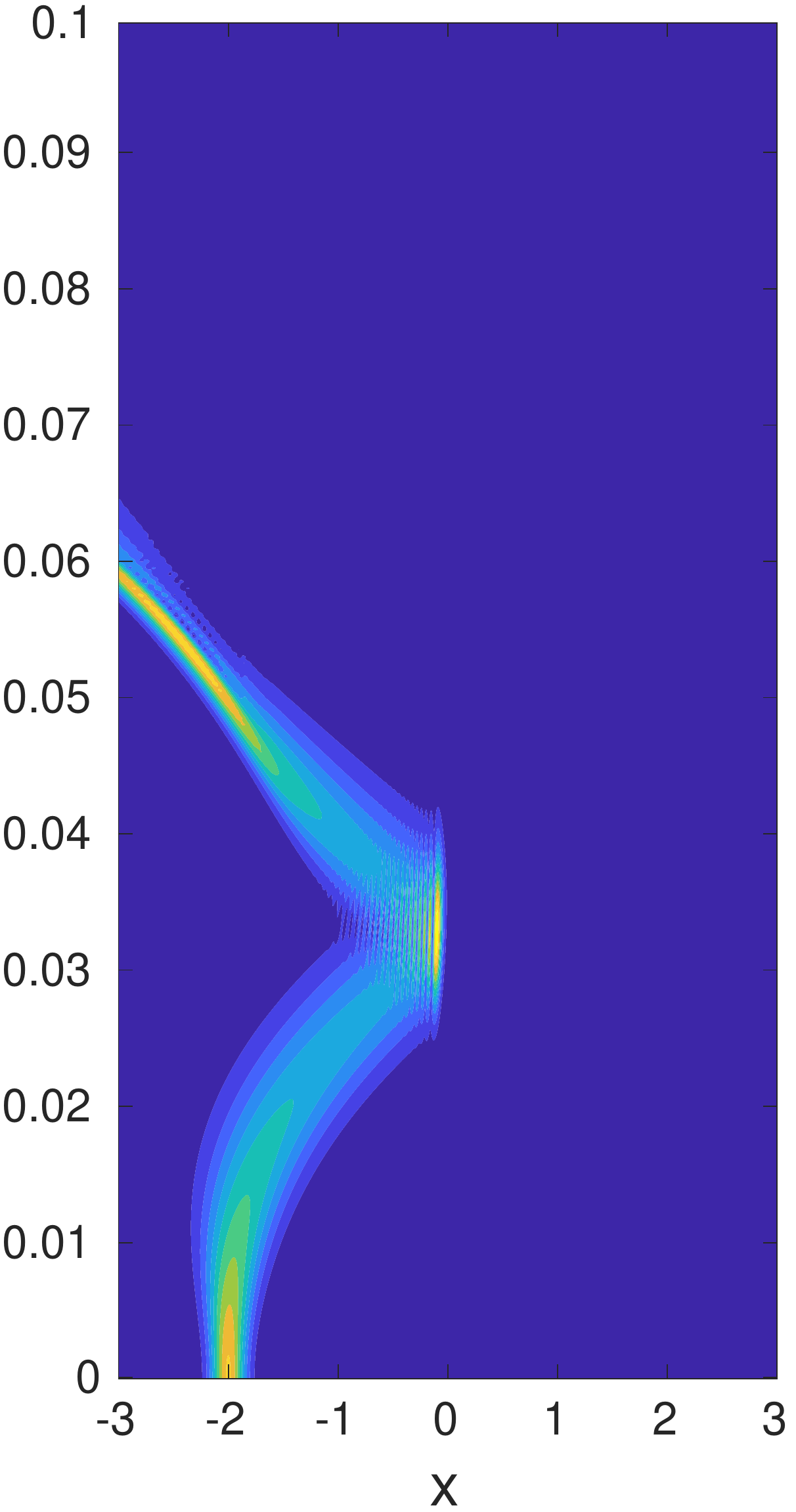}
  \hspace*{\fill}

  \caption{Plots of $A(t)$ and $|u(x,t)|$, displayed in
    the order $V_{\max} = 4000,6000,8000$, for Examples 2-4. The wavepacket
  is accelerated towards the potential barrier. Its
reflection coefficient depends on $V_{\max}$. After it interacts with the
barrier, the reflected and transmitted waves leave the domain.}
  \label{fig:ex2images}
\end{figure}

In this case, no analytical solution is available, so we measure the error
against highly accurate numerical solutions obtained using an
alternative discretization scheme, as follows. The
PDE is semidiscretized in time using the trapezoidal rule, and coupled
to the Robin condition \eqref{eq:sigmam}, with a very small time step
$\Delta t$. At each time step, this gives an elliptic two-point boundary
value problem with Robin conditions,
which we solve using a non-adaptive version of the high-order
scheme described in \cite{lee97}. We then sample this solution on the
Crank-Nicolson grid, and take it as our reference.

Before running any simulations, we compress the matrices $S_N$ and
$D_N$,
with several choices of $\Delta t$, for use in all three examples. We
repeat the convergence tests from Example 1 to demonstrate the
second-order accuracy of the overall scheme in $\Delta t$. The results
are shown in Table \ref{tab:ex2errvsdt}. As before, for each example, we
compare the TBC scheme to a Crank-Nicolson scheme with zero Dirichlet
boundary conditions on the domain $\Omega = [-25,25]$, which contains
the support of each solution to machine precision. The same
fixed $\Delta x$ is used in both cases. The errors are very similar for
both schemes. We speculate that the TBC scheme is more accurate than the brute
force scheme in Example 1 because the latter accumulates errors outside of
$\Omega$, which are then reintroduced once the sign of $A(t)$ reverses.
In Examples 2-4, the sign of $A(t)$ does not change, so the same
mechanism does not arise. This combination of results suggests that the
discretization error contributed by the TBC tends to be smaller than
that of the overall finite difference scheme.

\begin{table}[ht]
  \centering
  \begin{tabular}{|c|r|r|r|r|r|}
    \hline
    $\Delta t$ & $2 \times 10^{-5}$ & $1 \times 10^{-5}$ & $5 \times 10^{-6}$ & $2.5 \times
    10^{-6}$ & $1.25 \times 10^{-6}$ \\ \hline
    $V_{max} = 4000$, TBC & $6.56 \times 10^{-3}$ & $1.64 \times 10^{-3}$ &
    $4.11 \times 10^{-4}$ & $1.03 \times 10^{-4}$ & $2.58 \times
    10^{-5}$ \\ \hline
    $V_{max} = 4000$, DBC & $6.56 \times 10^{-3}$ & $1.64 \times 10^{-3}$ &
    $4.11 \times 10^{-4}$ & $1.03 \times 10^{-4}$ & $2.57 \times 10^{-5}$ \\ \hline
    $V_{max} = 6000$, TBC & $3.45 \times 10^{-1}$ & $8.77 \times 10^{-2}$ &
    $2.20 \times 10^{-2}$ & $5.60 \times 10^{-3}$ & $1.49 \times
    10^{-3}$ \\ \hline
    $V_{max} = 6000$, DBC & $3.45 \times 10^{-1}$ & $8.77 \times 10^{-2}$ &
    $2.20 \times 10^{-2}$ & $5.60 \times 10^{-3}$ & $1.49 \times 10^{-3}$ \\ \hline
    $V_{max} = 8000$, TBC & $6.45 \times 10^{-1}$ & $1.64 \times 10^{-1}$ &
    $ 4.12 \times 10^{-2}$ & $1.05 \times 10^{-2}$ & $2.80 \times
    10^{-3}$ \\ \hline
    $V_{max} = 8000$, DBC & $6.45 \times 10^{-1}$ & $1.64 \times 10^{-1}$ &
    $4.12 \times 10^{-2}$ & $1.05 \times 10^{-2}$ & $2.80 \times 10^{-3}$ \\
    \hline
  \end{tabular}
  \caption{Maximum $L^2$ errors on $[-3,3]$ versus $\Delta t$ for
    Examples 2-4. $\Delta x = 9.375 \times 10^{-6}$ is fixed throughout.}
  \label{tab:ex2errvsdt}
\end{table}

In Table \ref{tab:times2}, we give timings for Crank-Nicolson 
marching and for computing the TBCs using the butterfly scheme. Timings
are only shown for one of the three examples, since they are
all similar. The results demonstrate near-linear scaling of
the butterfly scheme. We note that in this case, 
the total time required to compute the
TBCs is significantly less than in the first example, remaining
under $0.02 \%$ of the total simulation cost for all choices of $N$. 
This is most likely a consequence of the simpler structure of $A(t)$, which 
leads to more compressible $S_N$ and $D_N$ matrices.

\begin{table}[ht]
  \centering
  \begin{tabular}{|l|r|r|r|r|r|r|}
    \hline
    \multicolumn{1}{|c|}{$N = T/\Delta t$} & 5,000 & 10,000 & 20,000 & 40,000 & 80,000 
    \\ \hline
    Time for finite difference marching & 167 & 340 & 677 & 1347 & 2679
     \\ \hline
    Time to obtain TBCs (butterfly) & 0.0255 & 0.0566 & 0.1226 &
     0.2563 & 0.5123
     \\
    \hline
  \end{tabular}
  \caption{Wall clock timings, in seconds, for the example with
  $V_{max} = 8000$; these are typical for Examples 2-4.}
  \label{tab:times2}
\end{table}

\section{Conclusions} \label{conclusions}

We have presented a derivation of simple TBCs for the \Schrod equation
with a vector potential in an arbitrary domain with a piecewise smooth
boundary.
For the one-dimensional case with $A(t) = 0$, these reduce to the well-known
exact nonlocal conditions. When $A(t) \neq 0$, the conditions become
non-convolutional, and we have developed a fast ``butterfly" scheme to
implement them. In order to reduce the computational complexity of the
precomputation phase of our scheme, it would be useful to explore recent variants of the
butterfly algorithm, such as \cite{hoying}, which address this issue.

In higher-dimensions, more
elaborate fast and memory-efficient algorithms will be required to
make exact TBCs conditions practical for
large-scale problems.
This work is in progress, and will be presented at a later date.

\section*{Acknowledgments} 

We would like to thank Alex Barnett, Angel Rubio, Umberto De Giovannini,
Hannes H\"ubener, Michael Ruggenthaler, and Mike O'Neil for many
useful conversations. Jason Kaye was supported in part by the Research
Training Group in Modeling and Simulation funded by the National Science
Foundation via grant RTG/DMS - 1646339.

\bibliographystyle{ieeetr}
{\footnotesize \bibliography{kg_tbcse1}}

\begin{thebibliography}{10}

\bibitem{ermolaev99}
A.~M. Ermolaev, I.~V. Puzynin, A.~V. Selin, and S.~I. Vinitsky, ``Integral
  boundary conditions for the time-dependent {S}chr{\"o}dinger equation: {A}tom
  in a laser field,'' {\em Phys. Rev. A}, vol.~60, no.~6, p.~4831, 1999.

\bibitem{baskakov91}
V.~A. Baskakov and A.~V. Popov, ``Implementation of transparent boundaries for
  numerical solution of the {S}chr{\"o}dinger equation,'' {\em Wave Motion},
  vol.~14, no.~2, pp.~123--128, 1991.

\bibitem{jiang04}
S.~Jiang and L.~Greengard, ``Fast evaluation of nonreflecting boundary
  conditions for the {S}chr{\"o}dinger equation in one dimension,'' {\em
  Comput. Math. Appl.}, vol.~47, no.~6, pp.~955--966, 2004.

\bibitem{lubich02}
C.~Lubich and A.~Sch\"{a}dle, ``Fast convolution for non–reflecting boundary
  conditions,'' {\em SIAM J. Sci. Comput.}, vol.~24, pp.~161--182, 2002.

\bibitem{schadle06}
A.~Sch{\"a}dle, M.~L{\'o}pez-Fern{\'a}ndez, and C.~Lubich, ``Fast and oblivious
  convolution quadrature,'' {\em SIAM J. Sci. Comput.}, vol.~28, no.~2,
  pp.~421--438, 2006.

\bibitem{hagstrom99}
T.~Hagstrom, ``Radiation boundary conditions for the numerical simulation of
  waves,'' {\em Acta Numer.}, vol.~8, pp.~47--106, 1999.

\bibitem{jiang08}
S.~Jiang and L.~Greengard, ``Efficient representation of nonreflecting boundary
  conditions for the time-dependent {S}chr{\"o}dinger equation in two
  dimensions,'' {\em Commun. Pure Appl. Math.}, vol.~61, no.~2, pp.~261--288,
  2008.

\bibitem{han04}
H.~Han and Z.~Huang, ``Exact artificial boundary conditions for the
  {S}chr{\"o}dinger equation in {$\mathbb{R}^2$},'' {\em Commun. Math. Sci.},
  vol.~2, no.~1, pp.~79--94, 2004.

\bibitem{han07}
H.~Han, D.~Yin, and Z.~Huang, ``Numerical solutions of {S}chr{\"o}dinger
  equations in {$\mathbb{R}^3$},'' {\em Numer. Methods Partial Differ. Equ.},
  vol.~23, no.~3, pp.~511--533, 2007.

\bibitem{antoine08}
X.~Antoine, A.~Arnold, C.~Besse, M.~Ehrhardt, and A.~Sch{\"a}dle, ``A review of
  transparent and artificial boundary conditions techniques for linear and
  nonlinear {S}chr{\"o}dinger equations,'' {\em Commun. Comput. Phys.}, vol.~4,
  no.~4, pp.~729--796, 2008.

\bibitem{schadle02}
A.~Sch{\"a}dle, ``Non-reflecting boundary conditions for the two-dimensional
  {S}chr{\"o}dinger equation,'' {\em Wave Motion}, vol.~35, no.~2,
  pp.~181--188, 2002.

\bibitem{feshchenko11}
R.~M. Feshchenko and A.~V. Popov, ``Exact transparent boundary condition for
  the parabolic equation in a rectangular computational domain,'' {\em J. Opt.
  Soc. Am. A}, vol.~28, pp.~373--380, 2011.

\bibitem{feshchenko13}
R.~M. Feshchenko and A.~V. Popov, ``Exact transparent boundary condition for
  the three-dimensional {S}chr\"odinger equation in a rectangular cuboid
  computational domain,'' {\em Phys. Rev. E}, vol.~88, p.~053308, 2013.

\bibitem{vaibhav15b}
V.~Vaibhav, ``On the nonreflecting boundary operators for the general two
  dimensional {S}chr{\"o}dinger equation,'' {\em arXiv}, vol.~1502.04519, 2015.

\bibitem{aes2003}
A.~Arnold, M.~Ehrhardt, and I.~Sofronov, ``Discrete transparent boundary
  conditions for the {S}chr{\"o}dinger equation: fast calculation,
  approximation and stability,'' {\em Commun. Math. Sci}, vol.~1, no.~3,
  pp.~501--556, 2003.

\bibitem{arnold12}
A.~Arnold, M.~Ehrhardt, M.~Schulte, and I.~Sofronov, ``Discrete transparent
  boundary conditions for the {S}chr{\"o}dinger equation on circular domains,''
  {\em Commun. Math. Sci.}, vol.~10, no.~3, pp.~889--916, 2012.

\bibitem{ji18}
S.~Ji, Y.~Yang, G.~Pang, and X.~Antoine, ``Accurate artificial boundary
  conditions for the semi-discretized linear {S}chr{\"o}dinger and heat
  equations on rectangular domains,'' {\em Comput. Phys. Commun.}, vol.~222,
  pp.~84--93, 2018.

\bibitem{vaibhav15}
V.~Vaibhav, ``Transparent boundary condition for numerical modeling of intense
  laser–molecule interaction,'' {\em J. Comput. Phys.}, vol.~283,
  pp.~478--494, 2015.

\bibitem{feshchenko17}
R.~M. Feshchenko and A.~V. Popov, ``Exact transparent boundary conditions for
  the parabolic wave equations with linear and quadratic potentials,'' {\em
  Wave Motion}, vol.~68, pp.~202--209, 2017.

\bibitem{antoine17}
X.~Antoine, E.~Lorin, and Q.~Tang, ``A friendly review of absorbing boundary
  conditions and perfectly matched layers for classical and relativistic
  quantum waves equations,'' {\em Mol. Phys.}, vol.~115, no.~15-16,
  pp.~1861--1879, 2017.

\bibitem{degiovannini15}
U.~De~Giovannini, A.~H. Larsen, and A.~Rubio, ``Modeling electron dynamics
  coupled to continuum states in finite volumes with absorbing boundaries,''
  {\em Eur. Phys. J. B}, vol.~88, no.~3, p.~56, 2015.

\bibitem{weinmuller17}
M.~Weinm{\"u}ller, M.~Weinm{\"u}ller, J.~Rohland, and A.~Scrinzi, ``Perfect
  absorption in {S}chr{\"o}dinger-like problems using non-equidistant complex
  grids,'' {\em J. Comput. Phys.}, vol.~333, pp.~199--211, 2017.

\bibitem{lorin09}
E.~Lorin, S.~Chelkowski, and A.~Bandrauk, ``Mathematical modeling of boundary
  conditions for laser-molecule time-dependent {S}chr{\"o}dinger equations and
  some aspects of their numerical computation—one-dimensional case,'' {\em
  Numer. Methods Partial Differ. Equ.}, vol.~25, no.~1, pp.~110--136, 2009.

\bibitem{economou}
E.~N. Economou, {\em Green's functions in Quantum Physics}.
\newblock Springer, Berlin, 2006.

\bibitem{pogorzelski}
W.~Pogorzelski, {\em Integral equations and their applications}.
\newblock Pergamon Press, 1966.

\bibitem{guentherlee1988}
R.~B. Guenther and J.~W. Lee, {\em Partial Differential Equations of
  Mathematical Physics and Integral Equations}.
\newblock Prentice Hall, 1988.

\bibitem{hairer85}
E.~Hairer, C.~Lubich, and M.~Schlichte, ``Fast numerical solution of nonlinear
  {V}olterra convolution equations,'' {\em SIAM J. Sci. Stat. Comp.}, vol.~6,
  no.~3, pp.~532--541, 1985.

\bibitem{michielssen}
E.~Michielssen and A.~Boag, ``A multilevel matrix decomposition algorithm for
  analyzing scattering from large structures,'' {\em IEEE Trans. Antennas
  Propag.}, vol.~44, p.~1086–1093, 1996.

\bibitem{oneil10}
M.~O'Neil, F.~Woolfe, and V.~Rokhlin, ``An algorithm for the rapid evaluation
  of special function transforms,'' {\em Appl. Comput. Harmon. Anal.}, vol.~28,
  no.~2, pp.~203--226, 2010.

\bibitem{candes09}
E.~Cand\`es, L.~Demanet, and L.~Ying, ``A fast butterfly algorithm for the
  computation of {F}ourier integral operators,'' {\em Multiscale Model.
  Simul.}, vol.~7, no.~4, pp.~1727--1750, 2009.

\bibitem{hoying}
Y.~Li, H.~Yang, E.~Martin, K.~L. Ho, and L.~Ying, ``Butterfly factorization,''
  {\em Multiscale Model. Simul.}, vol.~13, pp.~714--732, 2015.

\bibitem{lee97}
J.-Y. Lee and L.~Greengard, ``A fast adaptive numerical method for stiff
  two-point boundary value problems,'' {\em SIAM J. Sci. Comput.}, vol.~18,
  no.~2, pp.~403--429, 1997.

\end{thebibliography}

\end{document}